\documentclass[apaper,11pt]{article}
\usepackage{mathrsfs}
\usepackage{color}
\usepackage{amsmath, epsfig, cite}
\usepackage{amsthm}
\usepackage{amssymb}
\usepackage{amsfonts}
\usepackage{latexsym}
\usepackage{graphicx}
\usepackage{multirow}

\newtheorem{thm}{Theorem}[section]

\newtheorem{lem}{Lemma}[section]

\numberwithin{equation}{section}

\makeatletter \@addtoreset{equation}{section} \makeatother

\begin{document}


\begin{center}
{\Large \bf The g-extra connectivity of the Mycielskian}\footnote{Supported by the
NSFQH No.2019-ZJ-921; NSFC No.11661068.}

\vspace{3mm}

{\small \bf He Li}$^a$,\
{\small \bf Shumin Zhang}$^{b}$\
{\small \bf Chengfu Ye}$^{b}$\footnote{Corresponding author.},\

\vspace{3mm}

\baselineskip=0.20in
$^a${\it School of Computer,} \\
{\it Qinghai Normal University, Xining, 810001, China} \\
$^b${\it School of Mathematics and Statistics,} \\
{\it Qinghai Normal University, Xining, 810001, China} \\

{\tt lihe0520@yahoo.com, zhangshumin@qhnu.edu.cn, yechf@qhnu.edu.cn}

\end{center}

\baselineskip=0.20in
\begin{center}
{\bf Abstract }
\end{center}

\baselineskip=0.30in
The $g$-extra connectivity is an important parameter to measure the ability of tolerance and reliability of interconnection networks. Given a connected graph $G=(V,E)$ and a non-negative integer $g$, a subset $S\subseteq V$ is called a $g$-extra cut of $G$ if $G-S$ is disconnected and every component of $G-S$ has at least $g+1$ vertices. The cardinality of the minimum $g$-extra cut is defined as the $g$-extra connectivity of $G$, denoted by $\kappa_g(G)$. In a search for triangle-free graphs with arbitrarily large chromatic numbers, Mycielski developed a graph transformation that
transforms a graph $G$ into a new graph $\mu(G)$, which is called the Mycielskian of $G$. This paper investigates the relationship of the g-extra connectivity of the Mycielskian $\mu(G)$ and the graph $G$, moreover, show that $\kappa_{2g+1}(\mu(G))=2\kappa_{g}(G)+1$ for $g\geq 1$ and $\kappa_{g}(G)\leq min\{g+1, \lfloor\frac{n}{2}\rfloor\}$.



\baselineskip=0.30in
\section{Introduction}

With the fast advancements of multiprocessor systems, the topic of an interconnection network is an important research area. Furthermore, the fault-tolerance or reliability of a network are important. A network is usually represented by a graph where vertices represent processors and edges respresent communication links between \cite{WX2}. The fault-tolerance or reliability of a network are often measured by the connectivity of a correasponding graph. A (vertex) cut of $G$ is a set $S\subseteq G$ such that $G-S$ is disconnected or trivial. The connectivity of $G$, denoted by $\kappa(G)$, is defined as the minimum cardinality over all the cuts of $G$. However, the parameter tacitly assumes that all vertices that are
adjacent to, the same vertex can potentially fail simultaneously. This
is practically impossible in some network applications. To solve this problem, the $g$-extra connectivity is introduced.

In 1996, F\'{a}brega and Fiol \cite{FF} proposed the $g$-extra connectivity of a graph. Given a connected graph $G=(V,E)$ and a non-negative integer $g$, a subset $S\subseteq V$ is called a $g$-extra cut of $G$ if $G-S$ is disconnected and every component of $G-S$ has at least $g+1$ vertices. The cardinality of the minimum $g$-extra cut is defined as the $g$-extra connectivity of $G$, denoted by $\kappa_g(G)$. Note that $\kappa_0(G)=\kappa(G)$. In the study of the $g$-extra connectivity, much of the work has been focused on the computing some given networks for $g$ with smaller values (see, for example, \cite{FF, Boesch, Chang, CN, FJ, Esfahanian, EA, Hong, Hsieh, Latifi, Wan, Yang, Zhu, ZQ }). In fact, the computing of the $g$-extra connectivity is very difficult. As pointed out in \cite{Chang, EA}, no polynomial-time algorithm has been presented to compute $\kappa_{g}(G)$ for a general graph.

In a search for triangle-free graphs with arbitrarily large chromatic numbers, Mycielski \cite{Mycielski}  developed an interesting
graph transformation as follows. For a graph $G=(V ,E)$, the Mycielskian of $G$ is the graph $\mu(G)$ with the vertex set $V\cup V^{'}\cup \{u\}$, where $V^{'}=\{x^{'}:x\in V\}$ and the edge set $E\cup \{xy^{'}:xy\in E\}\cup \{y^{'}u:y^{'}\in V^{'}\}$. The vertex $x^{'}$ is called the twin
of the vertex $x$ (and $x$ is the twin of $x^{'}$). Furthermore, for any $F\subseteq V$, the set $F^{'}$ ($F^{'}\subseteq V^{'}$) is called the twin
of $F$ (and $F$ is the twin of $F^{'}$). The vertex $u$ is called the root of $\mu(G)$. For example, let $G$ be a path of order $n$, then the Mycielskian of $G$ is shown in Figure $1$. For $n\geq 2$, $\mu^{n}(G)$ is defined iteratively
by setting $\mu^{n}(G)=\mu(\mu^{n-1}(G))$.

\begin{figure}[!ht]
\begin{center}
\includegraphics[scale=0.75]{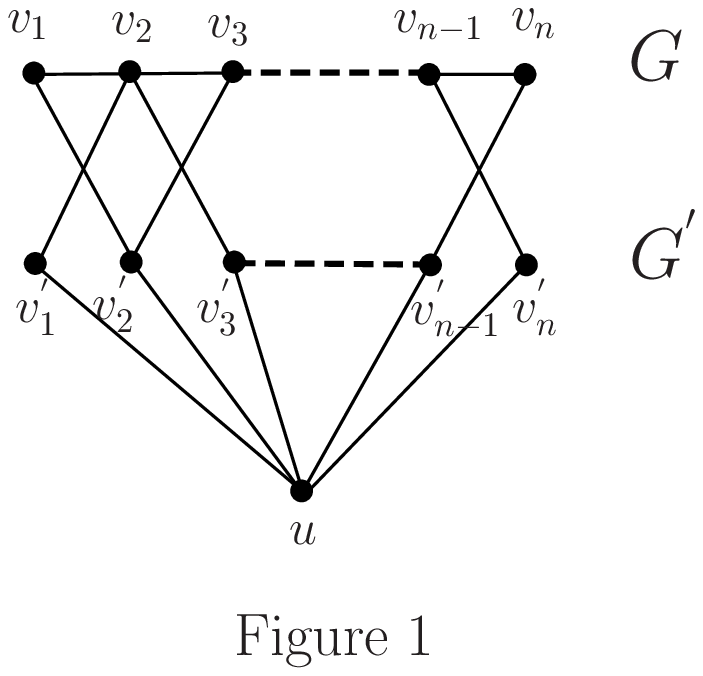}
\end{center}
\end{figure}

In recent times, there has been an increasing interest in the study of the Mycielskian, especially, in the study of their
circular chromatic numbers \cite{CGJ, Hajibolhassan, Lam, Liu, LiuH}. One of these papers is by Chang et al. \cite{CGJ}  wherein they have proved if $G$ has no isolated vertices, then $\kappa(\mu(G))\geq \mu(G)+1$. In 2008, R. Balakrishnan and S. Francis Raj \cite{Balakrishnan} have investigated the vertex-connectivity $\kappa(\mu(G))$ and edge-connectivity $\kappa^{'}(\mu(G))$ of $\mu(G)$ and obtained the
following results.

\begin{lem}{\upshape \cite{Balakrishnan}}\label{lem1-2}
If $G$ is a connected graph, then

$(1)$ $\kappa(\mu(G))=2\kappa(G)+1$ iff $\delta(G)\geq 2\kappa(G),$

$(2)$ $\kappa(\mu(G))=min\{\delta(G)+1, 2\kappa(G)+1\}$.
\end{lem}

In 2012, Guo and Liu \cite{Guo} have proposed that $\mu(G)$ is super-$\kappa$ if and only if $\delta(G) < 2\kappa(G)$, and $\mu(G)$ is super-$\lambda$
if and only if $G\ncong K_2$ for a connected graph $G$
with $|V(G)| \geq 2$.

In this paper, we investigate the relationship of the g-extra connectivity of the Mycielskian $\mu(G)$ and the graph $G$, moreover, show that $\kappa_{2g+1}(\mu(G))=2\kappa_{g}(G)+1$ for $g\geq 1$ and $\kappa_{g}(G)\leq g+1$.

\section{Terminology and notations}

All graphs considered in this paper are connected, undirected, finite and
simple. We refer to \cite{Bondy} for graph theoretical notation and
terminology not described here. For a graph $G$, let $V$, $E$, $e(G)$ and $n(G)$ denote the set of vertices, the set of edges, the size
and the order of $G$, respectively. A subgraph $H$ of $G$ is a graph with $V(H)\subseteq V(G)$,
$E(H)\subseteq E(G)$, and the endpoints of every edge in $E(H)$ belonging to $V(H)$. For any subset $X$ of $V(G)$, we use $G-X$ to denote the subgraph of G obtained by removing all the vertices
of $X$ together with the edges incident with them from $G$. If $X={v}$ , we simply write $G-v$ for $G-{X}$. The degree of a vertex $v$ in $G$, denoted by $deg_G(v)$, is the number of edges of $G$
incident with $v$. The neighbor set of a set $X\subseteq V$ (or a vertex $v$ ) in $G$ is denoted by $N_G(X)$ (or $N_G(v)$) and the $\delta(G)$ is minimum degree of the vertices of $G$.

\section{Main results}

In this section, we determine the relationship of the g-extra connectivity of the Mycielskian $\mu(G)$ and the graph $G$ for $g\geq 0$.

For a graph $G$, by $\kappa_0(G)=\kappa(G)$ and Lemma $1.1$, the following result is immediately obtained when $g=0$.

\begin{thm}\label{thm2-1}
If $G$ is a connected graph, then

$(1)$ $\kappa_{0}(\mu(G))=2\kappa_{0}(G)+1$ iff $\delta(G)\geq 2\kappa_{0}(G)$,

$(2)$ $\kappa_{0}(\mu(G))=min\{\delta(G)+1, 2\kappa_{0}(G)+1\}$.
\end{thm}

Furthermore, we can obtain the relationship of the g-extra connectivity of the Mycielskian $\mu(G)$ and the graph $G$ for $g\geq 1$.

\begin{thm}\label{thm2-2}
For a graph $G$, let $g$ be a non-negative integer $(g\geq 1)$ and $\kappa_{g}(G)\leq min\{g+1, \lfloor\frac{n}{2}\rfloor\}$, then
$$
\kappa_{2g+1}(\mu(G))=2\kappa_{g}(G)+1.
$$
\end{thm}
\begin{proof}

By the definition of $\kappa_g(G)$, there exists a set $F\subseteq V$ and $|F|=\kappa_g(G)$ whose deletion results that $G$ is disconnected and every remaining component has at least $g+1$ vertices. Clearly, from the structure of $\mu(G)$, there is the twin $F^{'}$ of $F$, and deleting $F\cup F^{'}\cup u$ in $\mu(G)$, we know $\mu(G)-(F\cup F^{'}\cup u)$ is disconnected and each remaining component has at least $2g+2$ vertices. Hence, $\kappa_{2g+1}(\mu(G))\leq 2\kappa_{g}(G)+1$.

To prove the converse, let $S$ be a arbitrary vertex set of $\mu(G)$ and $|S| \leq 2\kappa_{g}(G)$, we are to derive a contradiction. For the convenience of discussion,
let $S\cap V=A$ and $S\cap V{'}=B{'}$. The following two cases are considered.

\textbf{Case 1.} $u\notin S$.

\textbf{Case 1.1} $|A|< \kappa_{g}(G)$.

\textbf{Claim 1.} If $G-A$ is connected in $G$, then $\mu(G)-S$ is connected.

Let $M=G-A$ and $C^{'}=M^{'}+N_{A^{'}}(M)$, where $M^{'}$ and $A^{'}$ are the twins of $M$ and $A$, respectively. We consider the claim from the following situations.

\begin{itemize}

\item $C^{'}\nsubseteq B^{'}$.

     There exists at least one vertex $w\in C^{'}$ and $w\notin B^{'}$. Since $u$ is adjacent to all vertices of $V^{'}$ and $w$ is connected with $M$, we know $\mu(G)-S$ is connected.

\item  $C^{'}\subseteq B^{'}$.

From the structure of $\mu(G)$, it is easy to get that $\mu(G)-S$ is disconnected. Since $2\kappa_{g}(G)\geq |S|=|A+B^{'}|\geq |A+M^{'}+N_{A^{'}}(M)|\geq n+1$, contradicting that $\kappa_{g}(G)\leq \lfloor\frac{n}{2}\rfloor$.

\end{itemize}

\textbf{Claim 2.} If $G-A$ is disconnected in $G$ and there exists at least one component that has at most $g$ vertices, then $\mu(G)-S$ is connected, or $\mu(G)-S$ is disconnected and there exists at least one component that has at most $2g+1$ vertices.

Let $G_1,G_2,\cdots,G_r$ $(r\geq 2)$ are all components of $G-A$. We consider the claim from the following situations.

\begin{itemize}
\item $|G_i|\leq 2g+1$ for any $G_i$ $(1\leq i\leq r)$.

    If $G_i^{'}+N_{G^{'}}(G_i) \subseteq B^{'} $ ($G_i^{'}$ is the twin of $G_i$), then $\mu(G)-S$ is disconnected and $G_i$ is one component of $\mu(G)-S$. If $G_i^{'}+N_{G^{'}}(G_i) \nsubseteq B^{'}$, then there exists at least one vertex $w\in G_i^{'}+N_{G^{'}}(G_i)$ and $w\notin B^{'}$. Since $w$ is adjacent to $u$ and is connected with $G_i$, we know $\mu(G)-S$ is connected.

\item $|G_i|\geq 2g+2$ for some $G_i$ $(1\leq i\leq r)$.

     Without loss of generality, let $|G_1|\geq 2g+2$. If $(G_1^{'}+N_{G^{'}}(G_1)) \subseteq B^{'} $, then $\mu(G)-S$ is disconnected and $G_1$ is one component of $\mu(G)-S$. Since $2\kappa_{g}(G)\geq |S|=|A+B^{'}|\geq |A+G_1^{'}+N_{G^{'}}(G_1)|\geq 1+2g+2+1=2g+4$, contradicting that $\kappa_{g}(G)\leq g+1$. If $(G_1^{'}+N_{G^{'}}(G_1)) \nsubseteq B^{'}$, then $\mu(G)-S$ is either connected or disconnected and all components of $\mu(G)-S$ have at most $2g+1$ vertices.
\end{itemize}

From Claim $1$ and Claim $2$, when $|A|< \kappa_{g}(G)$, we can get $\mu(G)-S$ is connected, or $\mu(G)-S$ is disconnected and there exists at least one component that has at most $2g+1$ vertices, contradicting that the definition of $\kappa_{2g+1}(\mu(G))$.\qed

\item \textbf{Case 1.2} $|A|\geq \kappa_{g}(G)$.

\textbf{Claim 3.}  If $G-A$ is connected in $G$, then $\mu(G)-S$ is connected.

The proof of this claim is similar to Claim $1$.

\textbf{Claim 4.} If $G-A$ is disconnected in $G$ and there exists at least one component that has at most $g$ vertices, then $\mu(G)-S$ is connected, or $\mu(G)-S$ is disconnected and there exists at least one component that has at most $2g+1$ vertices.

The proof of this claim is similar to Claim $2$.

\textbf{Claim 5.} If $G-A$ is disconnected in $G$ and all components of $G-A$ have at least $g+1$ vertices, then $\mu(G)-S$ is connected.

Clearly, from the structure of $\mu(G)$, no matter how to delete $B^{'}$, we know $\mu(G)-S$ is always connected.

From Claim $3$, Claim $4$ and Claim $5$, when $|A|\geq \kappa_{g}(G)$, we can get $\mu(G)-S$ is connected, or $\mu(G)-S$ is disconnected and there exists at least one component that has at most $2g+1$ vertices, contradicting that the definition of $\kappa_{2g+1}(\mu(G))$.\qed

\textbf{Case 2.} $u\in S$.

\textbf{Case 2.1} $|A|< \kappa_{g}(G)$.

\textbf{Claim 6.} If $G-A$ is connected in $G$, then $\mu(G)-S$ is connected, or $\mu(G)-S$ is disconnected and the smallest component is an isolated vertex.

Let $M=G-A$. We consider the claim from the following two situations.

\begin{itemize}
\item $N_{A^{'}}(M)\subseteq B^{'}$.

From the structure of $\mu(G)$, it is easy to get that all vertices of $A^{'}-N_{A^{'}}(M)$ are isolated vertices in $\mu(G)-S$.

\item $N_{A^{'}}(M)\nsubseteq B^{'}$.

There exists at least one vertex $w\in N_{A^{'}}(M)$ and $w\notin B^{'}$. If $w$ is adjacent to all vertices of $A^{'}-N_{A^{'}}(M)$, then $w$ is connected with $M$. And $M^{'}$ is connected with $M$, so $\mu(G)-S$ is connected.
If $w$ is not adjacent to at least one vertex $v$ of $A^{'}-N_{A^{'}}(M)$, then $\mu(G)-S$ is disconnected and the smallest component is the vertex $v$.
\end{itemize}

\textbf{Claim 7.} If $G-A$ is disconnected in $G$ and there exists at least one component that has at most $g$ vertices, then $\mu(G)-S$ is connected, or $\mu(G)-S$ is disconnected and there exists at least one component that has at most $2g$ vertices.

Let $G_1,G_2,\cdots,G_r$ $(r\geq 2)$ be all components of $G-A$, and $X=\{G_i| |G_i|\leq g, 1 \leq i\leq k, 1 \leq k\leq r\}$, $Y=\{G_1, G_2, \cdots, G_r\}-X=\{G_j| |G_j|\geq g+1,k+1 \leq j\leq r\}$, $D^{'}=A^{'}\cap B^{'}$ and $F^{'}=A^{'}-D^{'}$. Three situations are considered.

\begin{itemize}
\item  $D^{'}=\varnothing$.

From the structure of $\mu(G)$, all components of $G-A$ are connected with $A^{'}$. Thus we can know $\mu(G)-S$ is connected.

\item  $D^{'}\neq\varnothing$ and $D^{'}=A^{'}$.

From the structure of $\mu(G)$, we have $\mu(G)-S$ is disconnected and $G_i+G^{'}_i$ $(G_i\in X)$ is one component of $\mu(G)-S$. Since $G_i\in X$ and $G_i^{'}$ is the twin of $G_i$, then $|G_i+G^{'}_i|\leq 2g$.

\item  $D^{'}\neq\varnothing$ and $D^{'}\neq A^{'}$.

\begin{itemize}
  \item[$\bullet$] $N_{A^{'}}(G_i)\subseteq D^{'}$ for some $G_i$ ($G_i\in X$).

       From the structure of $\mu(G)$, we have $\mu(G)-S$ is disconnected and $G_i+G^{'}_i$ is one component of $\mu(G)-S$. Since $G_i\in X$ and $G_i^{'}$ is the twin of $G_i$, then $|G_i+G^{'}_i|\leq 2g$.

  \item[$\bullet$] $N_{A^{'}}(G_i)\nsubseteq D^{'}$ for any $G_i$ ($G_i\in X$).

  If $N_{A^{'}}(G_j)\subseteq F^{'}$ for some $G_j$ ($G_j\in Y$), clearly, we know $\mu(G)-S$ is disconnected and includes the component $G_1+G_2+\cdots +G_k+G_j+G^{'}_1+G^{'}_2+\cdots +G_k^{'}+G_j^{'}$ ($1 \leq k\leq r$) and $|G_1+G_2+\cdots +G_k+G_j+G^{'}_1+G^{'}_2+\cdots +G_k^{'}+G_j^{'}|\geq 2g+3$, then $|D|=|D^{'}|=\kappa_{g}(G)$, contradicting that $|D|<|A|< \kappa_{g}(G)$.

 If $N_{A^{'}}(G_j)\nsubseteq F^{'}$ and $N_{A^{'}}(G_j)\cap F^{'}= \varnothing$ for any $G_j$ ($G_j\in Y$), clearly, we know $\mu(G)-S$ is disconnected and includes the component $G_1+G_2+\cdots +G_k+G^{'}_1+G^{'}_2+\cdots +G_k^{'}$ ($1 \leq k\leq r$). If $|G_1+G_2+\cdots +G_k|\geq g+1$, then $|D|=|D^{'}|=\kappa_{g}(G)$, contradicting that $|D|<|A|< \kappa_{g}(G)$. If $|G_1+G_2+\cdots +G_k|\leq g$, then $|G_1+G_2+\cdots +G_k+G^{'}_1+G^{'}_2+\cdots +G_k^{'}|\leq 2g$.

If $N_{A^{'}}(G_j)\nsubseteq F^{'}$ and $N_{A^{'}}(G_j)\cap F^{'}\neq \varnothing$ for any $G_j$ ($G_j\in Y$), then all components of $G-A$ are connected with $F^{'}$. Thus, we know $\mu(G)-S$ is connected.

\end{itemize}

From Claim $6$ and Claim $7$, when $|A|< \kappa_{g}(G)$, we can get $\mu(G)-S$ is connected, or $\mu(G)-S$ is disconnected and there exists at least one component that has at most $2g$ vertices, contradicting that the definition of $\kappa_{2g+1}(\mu(G))$.\qed

\textbf{Case 2.2} $|A|\geq \kappa_{g}(G)$.

\textbf{Claim 8.} If $G-A$ is connected in $G$, then $\mu(G)-S$ is connected, or $\mu(G)-S$ is disconnected and the smallest component is an isolated vertex.

The proof of this claim is similar to Claim $6$.

\textbf{Claim 9.} If $G-A$ is disconnected in $G$ and there exists at least one component that has at most $g$ vertices, then $\mu(G)-S$ is connected, or $\mu(G)-S$ is disconnected and there exists at least one component which has at most $2g$ vertices.

The proof of this claim is similar to Claim $7$.

\textbf{Claim 10.} If $G-A$ is disconnected in $G$ and every component of $G-A$ has at least $g+1$ vertices, then $\mu(G)-S$ is connected, or $\mu(G)-S$ is disconnected and the smallest component is an isolated vertex.

Let $G_1,G_2,\cdots,G_r$ $(r\geq 2)$ are all components of $G-A$ and $D^{'}=A^{'}\cap B^{'}$.

\item $D^{'}=\varnothing$.

From the structure of $\mu(G)$, all components of $G-A$ are connected with $A^{'}$, then $\mu(G)-S$ is connected.

\item  $D^{'}\neq\varnothing$ and $D^{'}=B^{'}$.

In fact, from $|V\cap S|=|A|\geq \kappa_{g}(G)$, so $|V^{'}\cap S|=|B^{'}|\leq \kappa_{g}(G)-1$, we know there exists one vertex $w\in A$ and $w\notin B$ (the twin of $B^{'}$).

If $w$ is connected with all components of $G-A$, then we have $w^{'}$ (the twin of $w$) is also connected with the components, so $\mu(G)-S$ is connected.

If $w$ is connected with some components of $G-A$, then every component of $G-B$ has at least $g+1$ vertices in $G$, clearly, we have  $|B|=|B^{'}|=\kappa_{g}(G)$, contradicting that $|B^{'}|\leq \kappa_{g}(G)-1$.

If $w$ is disconnected with all components of $G-A$, we have $w$ and $w^{'}$ are connected with $B$. From the structure of $\mu(G)$, we have $\mu(G)-S$ is disconnected and $w^{'}$ is an isolated vertex component in $\mu(G)-S$.

\item $D^{'}\neq\varnothing$ and $D^{'}\neq B^{'}$.

If $D^{'}$ is connected with all components of $G-A$, then $\mu(G)-S$ is connected.

If $D^{'}$ is connected with some components of $G-A$, from the structure of $\mu(G)$ then $D$ is connected with these components, then every component of $G-D$ has at least $g+1$ vertices in $G$, clearly, $|D|=\kappa_{g}(G)$, contradicting that $|D|<|B^{'}|\leq \kappa_{g}(G)-1$.

\end{itemize}

From Claim $8$, Claim $9$ and Claim $10$, when $|A|\geq \kappa_{g}(G)$, we can get $\mu(G)-S$ is connected, or $\mu(G)-S$ is disconnected and there exists at least one component that has at most $2g$ vertices, contradicting that the definition of $\kappa_{2g+1}(\mu(G))$.

\end{proof}

\section{Conclusion}

Looking for the $g$-extra connectivity of a graph is quite difficult. In fact, its existence is also an open problem until now. The $g$-good-neighbor connectivity of a graph is the same as it. In the forthcoming paper, we will investigated the $g$-good-neighbor connectivity of the Mycielskian.

\end{document}